\documentclass[11pt,leqno]{amsart}
\usepackage{amsmath,amssymb,amsthm,booktabs,geometry,enumitem}
\usepackage[colorlinks=true,linkcolor=blue,citecolor=blue,urlcolor=blue]{hyperref}
\numberwithin{equation}{section}
\newtheorem{theorem}{Theorem}
\newtheorem{corollary}[theorem]{Corollary}
\newtheorem{lemma}{Lemma}[section]
\newtheorem{proposition}[lemma]{Proposition}
\theoremstyle{remark}
\newtheorem{remark}[lemma]{Remark}
\newtheorem*{remark*}{Remark}
\newcommand{\Z}{\mathbb{Z}}
\newcommand{\R}{\mathbb{R}}

\newcommand{\C}{\mathbb{C}}
\newcommand{\G}{\Z[i]}
\newcommand{\cP}{\mathcal{P}}
\newcommand{\dd}{\,\mathrm{d}}
\newcommand{\fd}{\mathfrak{d}}
\newcommand{\fa}{\mathfrak{a}}
\newcommand{\fb}{\mathfrak{b}}
\newcommand{\fp}{\mathfrak{p}}
\newcommand{\fv}{\mathfrak{v}}
\newcommand{\Nm}{\mathrm{N}}
\newcommand{\inv}[1]{\overline{#1}}

\title{Solutions of $x_1^2+x_2^2-x_3^2=n^2$ with small $x_3$, revisited}
\author{Stephan Baier}
\address{Stephan Baier, Ramakrishna Mission Vivekananda Educational and Research Institute, Department of Mathematics, G.~T.~Road, PO Belur Math, Howrah, West Bengal 711202, India}
\email{stephanbaier2017@gmail.com}
\date{10 October 2026}

\begin{document}
\begin{abstract}
Let $r(N)$ be the number of representations of $N$ as a sum of two squares. For every integer $n\ge2$ we prove an asymptotic formula for
\[
\sum_{x\equiv e\bmod 2}\Phi\Big(\frac{x}{M}\Big)\,r(n^2+x^2),
\]
that is, for the number of integral solutions of $x_1^2+x_2^2-x_3^2=n^2$ with $x_3\equiv e\bmod 2$ in a range $|x_3|\ll M$, with smooth weight $\Phi$. The formula has an explicit secondary term and a power saving in the error term, and it is non-trivial as soon as $M\ge n^{5/6+\varepsilon}$, that is, $M\ge D^{1/2-1/12+\varepsilon}$ with $D=n^2$. The leading term is $\frac4\pi(\log n)\int\Phi(x/M)\dd x$, independently of the arithmetic of $n$; this corrects the main term stated in the author's paper in Ramanujan J.\ 63 (2024). The proof factorises $n+ix_3$ in $\Z[i]$, which leads to Kloosterman sums. These are estimated by Weil's bound, combined with a Plancherel argument that exploits the curvature of the underlying quadric.

This paper was written in collaboration with Claude (Opus~5.5, High), an AI model developed by Anthropic; see Subsection~\ref{sec:contrib} for the respective contributions.
\end{abstract}
\maketitle

\section{Introduction}
Friedlander and Iwaniec \cite{FrIw} studied the integral solutions of
\begin{equation}\label{FIeq}
x_1^2+x_2^2-x_3^2=D
\end{equation}
for square-free $D\equiv5\bmod 8$. They obtained an asymptotic formula, with main term proportional to $L(1,\chi_D)$, for the number of solutions with $x_3\asymp M$ and $M\ge D^{1/2-1/1332}$. The essential point is that $M$ may be smaller than $\sqrt D$. Their analysis leads to averages of Weyl sums for roots of quadratic congruences, which are handled by the spectral methods of \cite{DFI}.

In \cite{B24} the author treated the complementary case where $D=n^2$ is a square. The main term stated in \cite{B24} is incorrect (see Remark~\ref{corrig}). In this paper we give a new proof, which corrects the main term and improves the result in several respects. Here the quaternary form $x_1^2+x_2^2-x_3^2-n^2x_0^2$ is split over $\mathbb Q$, the surface \eqref{FIeq} carries rational lines, and the problem has a different nature:
\begin{itemize}
\item the main term contains a factor $\log n$, reflecting the pole of $\zeta(s)=L(s,\chi_{n^2})$ at $s=1$;
\item the error terms lead to Kloosterman sums instead of Weyl sums for quadratic roots.
\end{itemize}
Our tool is the factorisation of $n+ix_3$ in the Gaussian integers. Throughout, $r(N)=\#\{(x_1,x_2)\in\Z^2:x_1^2+x_2^2=N\}$ and $\chi$ is the non-principal character modulo $4$.

\begin{theorem}\label{mainthm}
Let $n\ge 2$ and $e\in\{0,1\}$, where $e=1$ if $n$ is even. Put $\epsilon=1$ if $n\equiv e\equiv1\bmod2$ and $\epsilon=0$ otherwise. Let $\Phi:\R\to\C$ be smooth with support in $[-2,2]$, let $1\le M\le n$, and let $\eta>0$. Then
\begin{equation}\label{mainformula}
\begin{split}
\sum_{\substack{x\in\Z\\ x\equiv e\bmod 2}}\Phi\Big(\frac xM\Big)\,r(n^2+x^2)
={}&\frac4\pi\int_{\R}\Phi\Big(\frac xM\Big)\log\sqrt{n^2+x^2}\,\dd x\\
&+\frac4\pi\Big(c_0-\frac{\epsilon}2\log2-\lambda(n)\Big)\int_{\R}\Phi\Big(\frac xM\Big)\dd x
+O\Big(n^{3/4+\eta}+\frac{n^{5/4+\eta}}{M^{1/2}}\Big),
\end{split}
\end{equation}
where
\[
\lambda(n)=\sum_{\substack{p^\nu\mid n,\ \nu\ge1\\ p>2}}\frac{\log p}{p^\nu},\qquad
c_0=\gamma+\frac{L'(1,\chi)}{L(1,\chi)}+\frac{\log2}{3}-2\frac{\zeta'(2)}{\zeta(2)}=2.19379629605\ldots,
\]
and the implied constant depends only on $\Phi$ and $\eta$.
\end{theorem}

The left-hand side of \eqref{mainformula} equals $\sum\Phi(x_3/M)$ over all $(x_1,x_2,x_3)\in\Z^3$ with $x_1^2+x_2^2-x_3^2=n^2$ and $x_3\equiv e\bmod 2$. The two main terms are of size $M\log n$ and $M\,O(1+\log\log n)$, respectively, since $\lambda(n)\ll\log\log n$. The error term is $O(Mn^{-\eta})$ as soon as $M\ge n^{5/6+3\eta}$.

\begin{corollary}\label{cor1}
Let $n\ge3$ be odd and let $\Phi$ be smooth with support in $[1,2]$. For $n^{5/6+\eta}\le M\le n$,
\[
\sum_{\substack{(x_1,x_2,x_3)\in\Z^3\\ x_1^2+x_2^2-x_3^2=n^2\\ 2\mid x_3}}\Phi\Big(\frac{x_3}{M}\Big)
=\frac4\pi\,(\log n)\int_\R\Phi\Big(\frac xM\Big)\dd x+O\big(M\log\log n\big).
\]
The same holds for $2\nmid x_3$, and the sum over all $x_3$ has main term $\frac8\pi(\log n)\int\Phi(x/M)\dd x$.
\end{corollary}

\begin{proof}
Use Theorem~\ref{mainthm} with $e=0$ (resp.\ $e=1$; for the sum over all $x_3$, add both). By the support of $\Phi$ we have $\log\sqrt{n^2+x^2}=\log n+O(M^2/n^2)$ there, and $\lambda(n)\le\sum_{p\mid n}\frac{\log p}{p-1}\ll\log\log n$.
\end{proof}

\begin{corollary}[All $n$, all $x_3$]\label{cor2}
Let $n\ge2$, write $v=v_2(n)$, and let $\Phi,M,\eta$ be as in Theorem~\ref{mainthm}. Then
\begin{multline*}
\sum_{x\in\Z}\Phi\Big(\frac xM\Big)r(n^2+x^2)=\frac8\pi\int_\R\Phi\Big(\frac xM\Big)\log\sqrt{n^2+x^2}\,\dd x\\
+\frac4\pi\Big(2c_0-2\lambda(n)-\big(2-\tfrac32\,2^{-v}\big)\log2\Big)\int_\R\Phi\Big(\frac xM\Big)\dd x+O\Big(n^{3/4+\eta}+\frac{n^{5/4+\eta}}{M^{1/2}}\Big).
\end{multline*}
\end{corollary}
\begin{proof}
\emph{Reduction to odd $n$.} If $n$ and $x$ are both even, then $r(n^2+x^2)=r((n/2)^2+(x/2)^2)$, because $r(4N)=r(N)$. Hence, with $n_t=n/2^t$ and $M_t=M/2^t$,
\[
\sum_{x}\Phi\Big(\frac xM\Big)r(n^2+x^2)=\sum_{t=0}^{v-1}\ \sum_{x'\ \mathrm{odd}}\Phi\Big(\frac{x'}{M_t}\Big)r(n_t^2+x'^2)+\sum_{x'}\Phi\Big(\frac{x'}{M_v}\Big)r(n_v^2+x'^2).
\]

\emph{Main terms.} Apply Theorem~\ref{mainthm} to each term:
\begin{itemize}
\item for $t<v$, with $(n_t,e)=(\text{even},1)$ and $\epsilon=0$;
\item for the last term, with $(n_v,e)=(\text{odd},0)$ and $(\text{odd},1)$, the latter with $\epsilon=1$.
\end{itemize}
Note that $\lambda(n_t)=\lambda(n)$. After substituting $x=2^tx'$, the $t$-th term has main term
\[
\frac4\pi2^{-t}\Big[\int\Phi(x/M)\log\sqrt{n^2+x^2}\,\dd x+(c_0-\lambda(n)-t\log2)I_0\Big],\qquad I_0=\int\Phi(x/M)\dd x .
\]
The last two terms are the same expression with $t=v$, the second of them with an extra $-\frac12\log2\cdot I_0$. The stated main term follows from
\[
\sum_{t<v}2^{-t}+2\cdot2^{-v}=2,\qquad \sum_{t<v}t2^{-t}+2v2^{-v}+\tfrac12 2^{-v}=2-\tfrac32 2^{-v}.
\]

\emph{Error terms.} If $M_t\ge1$, the error term of the $t$-th application is
\[
\ll 2^{-3t/4}n^{3/4+\eta}+2^{-t/4}n^{5/4+\eta}M^{-1/2},
\]
which sums to the claimed bound. If $M_t<1$, the corresponding sum contains at most the term $x'=0$, which is $\ll n^\eta$, and so does its main term.
\end{proof}

\begin{remark}[Correction of {\cite{B24}}]\label{corrig}
The main theorem of \cite{B24} (Theorem~1 there) asserts, for odd $n$, $\Phi$ supported in $[1,2]$ and $n^{7/8+\varepsilon}\le M\le n$,
\[
\sum_{\substack{(x_1,x_2,x_3)\in\Z^3\\ x_1^2+x_2^2-x_3^2=n^2,\ 2\mid x_3}}\Phi\Big(\frac{x_3}M\Big)=\frac{32}{\pi}\,\mathcal P(n)\,(\log n)\int_\R\Phi\Big(\frac xM\Big)\dd x+O\big(M(\log\log n)^{18}\big),
\]
with an explicit arithmetic factor $\mathcal P(n)$ (see \cite[(1.2)]{B24}). By Corollary~\ref{cor1}, the correct leading constant is $\frac4\pi$, independent of $n$; Table~\ref{tab} confirms this numerically. The discrepancy has four sources in \cite{B24}.
\begin{enumerate}[label=(\alph*)]
\item \emph{The archimedean integral} (\cite[(5.9)]{B24}). After the change of variables $(\sigma,\tau)=(x_2,x_3)$, the integral $\iint\Phi_3(\tau)(n^2-\sigma^2+\tau^2)^{-1/2}\dd\sigma\dd\tau$ was evaluated as $2\pi\hat\Phi_3(0)$, following \cite[p.~162]{FrIw}. Since $x_1>0$ and $x_2>0$ there, the integration covers a quarter of the circle $x_1^2+x_2^2=n^2+\tau^2$. The correct value is $\frac\pi2\hat\Phi_3(0)$, which gives a factor $4$.
\item \emph{The logarithm} (end of the proof of \cite[Proposition~4]{B24}). The main term carries $\log(2\sqrt X)=\frac12\log n+O(1)$, but it was replaced by $\log n$. This gives a factor $2$.
\item \emph{Coprimality conditions} (\cite[\S2]{B24}). Writing $d_1=q_1m_2$ and $d_2=q_2m_1$, and later $q_1=\alpha_1v_1$ and $q_2=\alpha_2v_2$, the conditions $(d_1,d_2)=1$ and $(q_1,q_2)=1$ require, in the final variables, $(\alpha_1,v_2m_1)=(\alpha_2,v_1m_2)=1$. These conditions are missing.
\item \emph{The divisibility $esv_1v_2\mid n$} (proof of \cite[Proposition~4]{B24}). In the evaluation of the leading coefficient this condition was weakened to $v_1v_2\mid n/e$.
\end{enumerate}
With (c) and (d) repaired, the parameter sum defining $\mathcal P(n)$ becomes $1$, in accordance with Corollary~\ref{cor1}. We checked this in exact rational arithmetic for twelve odd $n$ up to $429$, and symbolically in $p$ for all prime powers $n=p^a$ with $a\le12$.

The present proof does not use the parametrisation of \cite{B24}. Beyond the correction, it improves \cite{B24} in four ways:
\begin{itemize}
\item the range is $M\ge n^{5/6+\varepsilon}$ instead of $n^{7/8+\varepsilon}$;
\item the error term has a power saving instead of $O(M(\log\log n)^{18})$;
\item the secondary term is explicit;
\item all $n$ and both parities of $x_3$ are covered.
\end{itemize}
The averages of Kloosterman sums discussed in \cite[Remark~3]{B24} are not needed. The gain over \cite{B24} comes from estimating the dual sum in $L^2$ (Lemma~\ref{planch}). A spectral method, developed in \cite{BSp}, reaches the range $M\ge n^{71/96+\varepsilon}$ and also treats the equation $x_1^2+x_2^2-x_3^2=D$ for non-square $D$.
\end{remark}

\subsection{Contributions of the author and of Claude}\label{sec:contrib}
This paper is the outcome of a collaboration between the author and Claude (Opus~5.5, High), an AI model developed by Anthropic. The two roles were as follows.

\emph{Stephan Baier.}
\begin{itemize}
\item Wrote the original paper \cite{B24}.
\item Raised the doubts about its main term.
\item Asked for a careful check of \cite{B24}, including whether modern results on averages of Kloosterman sums could improve its error term.
\item Asked for the correction and improvement of \cite{B24}, and for the statement of the correction in this paper.
\item Fixed the standards (every claim proved or explicitly flagged) and evaluated the successive drafts.
\end{itemize}

\emph{Claude.}
\begin{itemize}
\item Checked \cite{B24} line by line and located the four sources of error in its main term (Remark~\ref{corrig}). Confirmed the correct constant by exact computation and numerically.
\item Proposed the Gaussian-integer reformulation (Section~\ref{sec:alg}). Proved the main term with its secondary term (Section~\ref{sec:main}).
\item Proved the error term via reciprocity, Poisson summation, Weil's bound and the Plancherel step of Lemma~\ref{planch} (Section~\ref{sec:err}).
\item Carried out the numerical checks (Table~\ref{tab}) and wrote the \LaTeX{} text.
\end{itemize}

Results from the literature are cited and not re-proved here. The author takes responsibility for the final version.

\subsection*{Plan of the proof} The proof has four steps.
\begin{enumerate}[leftmargin=*]
\item \emph{Algebraic reduction} (Section~\ref{sec:alg}). $r(n^2+x^2)$ counts ideals $\fd\mid(n+ix)$ whose complementary divisor is coprime to its conjugate (Lemma~\ref{bij}). A Möbius inversion over rational integers and a smooth hyperbola device express this through primitive Gaussian divisors $a+bi$ of norm $q\ll n$. The condition $(a+bi)\mid(n+ix)$ becomes the congruence $x\equiv nb\inv a\bmod q$ (Proposition~\ref{identity}).
\item \emph{Poisson summation in $x$} (Section~\ref{sec:main}). This separates a main term, evaluated by counting primitive Gaussian integers (Section~\ref{sec:count}), from an error term containing the exponentials $e(hnb\inv{2a}/q)$.
\item \emph{Reciprocity and Poisson summation in $b$} (Section~\ref{sec:err}). This turns the error term into Kloosterman sums to moduli $\asymp a\le\sqrt{2q}$.
\item \emph{Estimation} (Section~\ref{sec:err}). We use Weil's bound. The decisive point is to sum the dual variable $l$ in $L^2$ (Plancherel) rather than pointwise. This is where the curvature of the quadric enters, and it improves the exponent $7/8$ to $5/6$.
\end{enumerate}

\section{Notation and tools}\label{sec:not}
We write $e(x)=e^{2\pi i x}$ and $\hat f(\xi)=\int_\R f(x)e(-x\xi)\dd x$. Poisson summation in the form
\begin{equation}\label{poisson}
\sum_{y\equiv y_0\bmod r}F(y)=\frac1r\sum_{h\in\Z}\hat F\Big(\frac hr\Big)e\Big(\frac{hy_0}r\Big)
\end{equation}
is used for smooth compactly supported $F$. For coprime $u,v$, $\inv u$ denotes an inverse of $u$ modulo $v$ (the modulus is always specified). The \emph{reciprocity law}
\begin{equation}\label{recip}
\frac{\inv u}{v}+\frac{\inv v}{u}\equiv\frac1{uv}\bmod 1\qquad(u,v\ge1,\ (u,v)=1)
\end{equation}
holds because $u\inv u+v\inv v\equiv1\bmod uv$. The Kloosterman sum $S(u,v;c)=\sum^*_{x\bmod c}e((ux+v\inv x)/c)$ satisfies Weil's bound \cite[Cor.~11.12]{IK}
\begin{equation}\label{weil}
|S(u,v;c)|\le\tau(c)\,(u,v,c)^{1/2}c^{1/2}.
\end{equation}

We work in $\G$. For a non-zero ideal $\fv$, $\Nm\fv$ denotes its norm, $\tau(\fv)$ the number of ideals dividing $\fv$, $v_{1+i}(\fv)$ the exponent of $(1+i)$ in $\fv$, and $\fv^\flat=\fv\,(1+i)^{-v_{1+i}(\fv)}$ its odd part. An ideal is \emph{primitive} if it is not divisible by $(p)$ for any rational prime $p$.

We fix once and for all a smooth $W:(0,\infty)\to[0,1]$ with $W(t)=1$ for $t\le\frac12$, $W(t)=0$ for $t\ge2$, and
\begin{equation}\label{Wsym}
W(t)+W(1/t)=1\qquad(t>0).
\end{equation}
Such a $W$ exists. Take a smooth $V_0:\R\to[0,1]$ equal to $1$ on $(-\infty,-\log2]$ and $0$ on $[\log2,\infty)$, set $V(u)=\frac12(V_0(u)+1-V_0(-u))$ and $W(t)=V(\log t)$. Two consequences of \eqref{Wsym} will be used: $W'(t)=t^{-2}W'(1/t)$, and hence
\begin{equation}\label{Wlog}
\int_0^\infty W'(u)\log u\,\dd u=0 .
\end{equation}
Indeed, substituting $u\mapsto1/u$ turns the integral into its own negative.

Throughout, $\eta>0$ is arbitrarily small, implied constants may depend on $\eta,\Phi,W$, and a quantity is \emph{negligible} if it is $O_A(n^{-A})$ for every $A>0$.

\section{Algebraic reduction}\label{sec:alg}

\begin{lemma}\label{bij}
Let $w\in\G\setminus\{0\}$. Then
\[
r(\Nm w)=4\cdot\#\big\{\fd\mid(w):\ (\fa,\bar\fa)=1,\ \text{where }\fa=(w)\fd^{-1}\big\}.
\]
\end{lemma}
\begin{proof}
Since every ideal of $\G$ is principal with four generators, $r(\Nm w)=4\#\{\fb:\Nm\fb=\Nm w\}$. We show that $\fd\mapsto\fb=\fd\bar\fa$ is a bijection from the set on the right onto $\{\fb:\Nm\fb=\Nm w\}$.

\emph{The map is well defined and injective.} $\Nm(\fd\bar\fa)=\Nm\fd\,\Nm\fa=\Nm w$. Moreover $\gcd(\fd\bar\fa,(w))=\fd\gcd(\bar\fa,\fa)=\fd$, so $\fd$ is recovered from $\fb$.

\emph{The map is surjective.} Given $\fb$ with $\Nm\fb=\Nm w$, let $\fd=\gcd(\fb,(w))$, and write $(w)=\fd\fa$, $\fb=\fd\fb'$ with $(\fa,\fb')=1$. Then $\Nm\fa=\Nm\fb'$, so $\fa\bar\fa=\fb'\bar\fb'$. Since $\fa$ is coprime to $\fb'$, we get $\fa\mid\bar\fb'$, and comparing norms gives $\fa=\bar\fb'$. Hence $\fb=\fd\bar\fa$ and $(\fa,\bar\fa)=(\fa,\fb')=1$.
\end{proof}

\begin{lemma}\label{coprimeconj}
An ideal $\fa\neq0$ satisfies $(\fa,\bar\fa)=1$ if and only if $(1+i)\nmid\fa$ and $(p)\nmid\fa$ for every odd rational prime $p$.
\end{lemma}
\begin{proof}
A prime ideal $\fp$ divides both $\fa$ and $\bar\fa$ if and only if $\fp\bar\fp\mid\fa$ or $\fp=\bar\fp\mid\fa$. The prime ideals are of three kinds:
\begin{itemize}
\item $(1+i)$, which is self-conjugate;
\item $(p)$ for $p\equiv3\bmod4$, which is self-conjugate;
\item $\pi,\bar\pi$ with $\pi\bar\pi=(p)$ for $p\equiv1\bmod 4$, which are distinct conjugates.
\end{itemize}
Hence $(\fa,\bar\fa)\neq1$ if and only if $(1+i)\mid\fa$ or $(p)\mid\fa$ for some odd prime $p$.
\end{proof}

\begin{proposition}\label{prop33}
Let $n,x\in\Z$, not both even, and $w=n+ix$. Then $v_{1+i}(w)\in\{0,1\}$, with $v_{1+i}(w)=1$ if and only if $n,x$ are both odd, and
\[
r(n^2+x^2)=4\sum_{\substack{c\ \mathrm{odd}\\ c\mid(n,x)}}\mu(c)\,\tau\big((w/c)^\flat\big).
\]
\end{proposition}
\begin{proof}
The first claim follows from $\Nm w=n^2+x^2\equiv1$ or $2\bmod4$. For $\fd\mid(w)$ put $\fa=(w)\fd^{-1}$. By Lemmas~\ref{bij} and \ref{coprimeconj}, $r(n^2+x^2)/4$ counts the $\fd\mid(w)$ such that no $(p)$ with $p$ odd divides $\fa$, and $(1+i)\nmid\fa$.

\emph{The odd primes.} Detect the first condition by $\sum_{c\ \mathrm{odd},\ (c)\mid\fa}\mu(c)$. Since $(c)\mid\fa$ if and only if $(c)\fd\mid(w)$, this requires $c\mid w$, that is, $c\mid(n,x)$, and then it says $\fd\mid(w/c)$.

\emph{The prime $(1+i)$.} The second condition means $v_{1+i}(\fd)=v_{1+i}(w)=v_{1+i}(w/c)$. The map $\fd\mapsto\fd(1+i)^{-v_{1+i}(\fd)}$ is a bijection from the divisors $\fd$ of $(w/c)$ with $v_{1+i}(\fd)=v_{1+i}(w/c)$ onto the divisors of $(w/c)^\flat$.
\end{proof}

\begin{lemma}\label{hyp}
For every ideal $\fv\ne0$,
\[
\tau(\fv)=2\sum_{\fd\mid\fv}W\Big(\frac{\Nm\fd}{\sqrt{\Nm\fv}}\Big).
\]
\end{lemma}
\begin{proof}
The involution $\fd\mapsto\fv\fd^{-1}$ maps $\Nm\fd/\sqrt{\Nm\fv}$ to its reciprocal. Hence $\sum_{\fd\mid\fv}W(\Nm\fd/\sqrt{\Nm\fv})=\sum_{\fd\mid\fv}W(\sqrt{\Nm\fv}/\Nm\fd)$. Adding the two expressions and using \eqref{Wsym} gives $\tau(\fv)$.
\end{proof}

Let
\[
\cP=\{(a,b)\in\Z^2:\ a\ge1,\ b\ge0,\ \gcd(a,b)=1,\ a\not\equiv b\bmod 2\},\qquad q=q(a,b)=a^2+b^2 .
\]

\begin{lemma}\label{prim}
Every ideal $\fd$ of odd norm can be written uniquely as $\fd=(k)(a+bi)$ with $k\ge1$ odd and $(a,b)\in\cP$, and then $\Nm\fd=k^2q$. The ideals $(a+bi)$ with $(a,b)\in\cP$ are exactly the primitive ideals of odd norm, each listed once.
\end{lemma}
\begin{proof}
\emph{Decomposition.} Let $k$ be the largest positive integer with $(k)\mid\fd$. Then $k$ is odd, and $\fd(k)^{-1}$ is primitive and determines $k$ uniquely.

\emph{Primitive ideals.} A principal ideal $(a+bi)$ is primitive if and only if $\gcd(a,b)=1$. Its norm $a^2+b^2$ is odd if and only if $a\not\equiv b\bmod 2$.

\emph{Generators.} Of the four generators $\pm(a+bi),\pm i(a+bi)$ of a non-zero ideal, exactly one lies in $\{a\ge1,b\ge0\}$.
\end{proof}

\begin{lemma}\label{cong}
Let $a,b\in\Z$ with $\gcd(a,b)=1$ and $q=a^2+b^2$ odd, and let $m,y\in\Z$. Then
\[
(a+bi)\mid(m+iy)\iff y\equiv mb\inv a\bmod q .
\]
\end{lemma}
\begin{proof}
We have $\gcd(a,q)=\gcd(a,b^2)=1$ and similarly $\gcd(b,q)=1$. Since $(a+bi)$ is primitive of norm $q$, the map $\Z/q\Z\to\G/(a+bi)$ is an isomorphism. In $\G/(a+bi)$ we have $i\equiv-a\inv b$, so $m+iy\equiv m-a\inv b y$, which vanishes if and only if $y\equiv mb\inv a\bmod q$. (For $q=1$ both sides are always true.)
\end{proof}

\begin{proposition}\label{identity}
Let $n,e,\epsilon$ be as in Theorem~\ref{mainthm} and $x\equiv e\bmod 2$. Then
\[
r(n^2+x^2)=8\sum_{\substack{j\ \mathrm{odd}\\ j\mid (n,x)}}\ \sum_{ck=j}\mu(c)\sum_{\substack{(a,b)\in\cP\\ (a+bi)\mid (n+ix)/j}}W\Big(\frac{kq\,2^{\epsilon/2}j}{\sqrt{n^2+x^2}}\Big).
\]
\end{proposition}
\begin{proof}
\emph{Step 1: Proposition~\ref{prop33}.} Note that $v_{1+i}(n+ix)=\epsilon$. Apply Proposition~\ref{prop33}, and then Lemma~\ref{hyp} to $\fv=(w/c)^\flat$, which has norm $(n^2+x^2)/(c^22^{\epsilon})$.

\emph{Step 2: Lemma~\ref{prim}.} The divisors of $\fv$ are the odd-norm divisors $\fd$ of $(w/c)$. Write them as $\fd=(k)(a+bi)$ by Lemma~\ref{prim}. Then $\fd\mid(w/c)$ if and only if $ck\mid(n,x)$ and $(a+bi)\mid w/(ck)$. Moreover
\[
\Nm\fd/\sqrt{\Nm\fv}=k^2q\,c\,2^{\epsilon/2}/\sqrt{n^2+x^2}.
\]

\emph{Step 3: collect.} Put $j=ck$, which is odd.
\end{proof}

We now insert Proposition~\ref{identity} into the sum in \eqref{mainformula}. Put $m=n/j$. For $j\mid n$ we write $x=jy$; since $j$ is odd, $x\equiv e\bmod 2$ if and only if $y\equiv e\bmod2$, and $\sqrt{n^2+x^2}=j\sqrt{m^2+y^2}$. By Lemma~\ref{cong} the divisibility condition reads $y\equiv mb\inv a\bmod q$. Hence the left-hand side $S_e$ of \eqref{mainformula} is
\begin{equation}\label{Se}
S_e=8\sum_{\substack{j\mid n\\ j\ \mathrm{odd}}}\ \sum_{ck=j}\mu(c)\sum_{(a,b)\in\cP}\ \sum_{\substack{y\equiv e\bmod 2\\ y\equiv mb\inv a\bmod q}}F_{q}(y),
\qquad
F_q(y)=\Phi\Big(\frac{jy}M\Big)W\Big(\frac{kq\,2^{\epsilon/2}}{\sqrt{m^2+y^2}}\Big),
\end{equation}
where $F_q=F_{j,k,q}$ also depends on $j$ and $k$. Since $|y|\le2M/j\le2m$, we have $F_q=0$ unless
\begin{equation}\label{qmax}
q\le \frac{2\sqrt5\,m}{k}\le\frac{5m}{k}.
\end{equation}
All sums in \eqref{Se} are therefore finite.

\section{Primitive Gaussian integers of odd norm}\label{sec:count}

\begin{lemma}\label{count}
For $X\ge1$,
\[
\#\{(a,b)\in\cP:q\le X\}=\frac X\pi+O\big(X^{1/2}\log 2X\big).
\]
\end{lemma}
\begin{proof}
\emph{Without coprimality.} Let $N_0(Y)$ count the pairs $(a,b)$ with $a\ge1$, $b\ge0$, $a\not\equiv b\bmod2$ and $a^2+b^2\le Y$. These pairs form a translate of the lattice $\{a\equiv b\bmod 2\}$, which has covolume $2$, intersected with a convex region of area $\pi Y/4+O(\sqrt Y)$ and perimeter $O(\sqrt Y)$. Hence $N_0(Y)=\pi Y/8+O(\sqrt Y+1)$.

\emph{Möbius inversion.} If $a\not\equiv b\bmod 2$, then $d=\gcd(a,b)$ is odd and $(a/d,b/d)\in\cP$. So $N_0(Y)=\sum_{d\ \mathrm{odd}}P(Y/d^2)$, where $P$ is the counting function in question, and therefore
\[
P(X)=\sum_{d\le\sqrt X,\ d\ \mathrm{odd}}\mu(d)N_0(X/d^2).
\]
Since $\sum_{d\ \mathrm{odd}}\mu(d)d^{-2}=\prod_{p>2}(1-p^{-2})=8/\pi^2$, the claim follows.
\end{proof}

\begin{lemma}\label{dirichlet}
For $\Re s>1$,
\[
D(s):=\sum_{(a,b)\in\cP}q^{-s}=\frac{\zeta(s)L(s,\chi)}{(1+2^{-s})\zeta(2s)}.
\]
\end{lemma}
\begin{proof}
By Lemma~\ref{prim}, $\sum_{\Nm\fd\ \mathrm{odd}}\Nm\fd^{-s}=\big(\sum_{k\ \mathrm{odd}}k^{-2s}\big)D(s)=(1-2^{-2s})\zeta(2s)D(s)$. On the other hand, removing the Euler factor of $(1+i)$ from the Dedekind zeta function of $\mathbb Q(i)$ gives $(1-2^{-s})\zeta(s)L(s,\chi)$ for the left-hand side. Divide.
\end{proof}

\begin{lemma}\label{harmonic}
Let $\gamma_0=c_0/\pi$ with $c_0$ as in Theorem~\ref{mainthm}. Then:
\begin{enumerate}[label=(\roman*)]
\item for $X\ge1$, $\displaystyle\sum_{(a,b)\in\cP,\ q\le X}\frac1q=\frac1\pi\log X+\gamma_0+O\big(X^{-1/2}\log2X\big)$;
\item for every $K>0$, $\displaystyle\sum_{(a,b)\in\cP}\frac{W(Kq)}{q}=\frac1\pi\log\frac1K+\gamma_0+O\big(K^{1/2}\log(2+K^{-1})\big)$.
\end{enumerate}
\end{lemma}
\begin{proof}
\emph{(i)} Write $P(t)=t/\pi+E(t)$, with $E(t)\ll t^{1/2}\log2t$ by Lemma~\ref{count}. Partial summation gives
\[
\sum_{q\le X}\frac1q=\frac{P(X)}X+\int_1^X\frac{P(t)}{t^2}\dd t=\frac1\pi\log X+C+O(X^{-1/2}\log2X),\qquad C:=\frac1\pi+\int_1^\infty\frac{E(t)}{t^2}\dd t.
\]
For real $s>1$, likewise $D(s)=s\int_1^\infty P(t)t^{-s-1}\dd t=\frac{1}{\pi(s-1)}+\frac1\pi+s\int_1^\infty E(t)t^{-s-1}\dd t$. Hence, by dominated convergence, $D(s)-\frac1{\pi(s-1)}\to C$ as $s\to1^+$.

On the other hand, by Lemma~\ref{dirichlet}, $D(s)=\zeta(s)G(s)$ with $G(s)=L(s,\chi)/((1+2^{-s})\zeta(2s))$ holomorphic near $1$. Here
\[
G(1)=\frac{\pi/4}{(3/2)(\pi^2/6)}=\frac1\pi,\qquad \frac{G'}{G}(1)=\frac{L'}{L}(1,\chi)+\frac{\log2}{3}-2\frac{\zeta'(2)}{\zeta(2)}.
\]
Since $\zeta(s)=\frac1{s-1}+\gamma+O(s-1)$, we get $D(s)-\frac{G(1)}{s-1}\to G(1)\big(\gamma+\frac{G'}{G}(1)\big)=c_0/\pi$. Hence $C=\gamma_0$.

\emph{(ii)} Let $S(t)=\sum_{q\le t}1/q$, so $S(t)=0$ for $t<1$. Stieltjes integration by parts gives $\sum W(Kq)/q=-\int_1^\infty S(t)\,KW'(Kt)\dd t$.

\emph{Case $K\le\frac12$.} Here $W'(Kt)=0$ for $t\le1$. Insert (i):
\begin{itemize}
\item the constant $\gamma_0$ contributes $\gamma_0W(K)=\gamma_0$;
\item the term $\frac1\pi\log t$ contributes, after substituting $u=Kt$, the amount
\[
\frac1\pi\log\frac1K-\frac1\pi\int_0^\infty W'(u)\log u\,\dd u=\frac1\pi\log\frac1K
\]
by \eqref{Wlog};
\item the remainder contributes $\ll\int_{1/(2K)}^{2/K}t^{-1/2}\log(2t)\,K\dd t\ll K^{1/2}\log(2/K)$.
\end{itemize}

\emph{Case $K>\frac12$.} The left-hand side is $\le\sum_{q\le 4}1/q\ll1$, and $|\frac1\pi\log K|+\gamma_0\ll K^{1/2}$, so (ii) holds trivially.
\end{proof}

\section{Poisson summation and the main term}\label{sec:main}
Fix an odd $j\mid n$ and a factorisation $ck=j$, and put $m=n/j$. For $(a,b)\in\cP$ and $h\in\Z$ let
\[
G_h(q)=\int_\R F_q(y)\,e\Big(-\frac{hy}{2q}\Big)\dd y .
\]

\begin{lemma}\label{poissonlemma}
For $(a,b)\in\cP$,
\[
\sum_{\substack{y\equiv e\bmod 2\\ y\equiv mb\inv a\bmod q}}F_q(y)=\frac1{2q}\sum_{h\in\Z}(-1)^{he}\,G_h(q)\,e\Big(\frac{hmb\,\inv{2a}}{q}\Big),
\]
where $\inv{2a}$ is the inverse of $2a$ modulo $q$.
\end{lemma}
\begin{proof}
\emph{Substitution.} Write $y=2z+e$, and let $\inv2$ be the inverse of $2$ modulo $q$, represented by the integer $(q+1)/2$. The two congruences become $z\equiv z_0:=\inv 2(mb\inv a-e)\bmod q$.

\emph{Poisson.} By \eqref{poisson} with modulus $q$, and since $\int F_q(2z+e)e(-hz/q)\dd z=\frac12e(\frac{he}{2q})G_h(q)$, the sum equals
\[
\frac1{2q}\sum_h e\Big(\frac{he}{2q}-\frac{he\,\inv2}{q}\Big)G_h(q)\,e\Big(\frac{h\inv2mb\inv a}{q}\Big).
\]

\emph{Simplification.} Here $\frac{e}{2q}-\frac{e(q+1)}{2q}=-\frac e2$, so the first factor is $(-1)^{he}$. Finally $\inv2\cdot\inv a\equiv\inv{2a}\bmod q$.
\end{proof}

By \eqref{Se} and Lemma~\ref{poissonlemma}, $S_e=\mathcal M_e+\mathcal E_e$. The main term $\mathcal M_e$ collects the terms with $h=0$:
\begin{equation}\label{Me}
\mathcal M_e=8\sum_{j}\sum_{ck=j}\mu(c)\sum_{(a,b)\in\cP}\frac{G_0(q)}{2q}.
\end{equation}
The error term collects the rest:
\begin{equation}\label{Ee}
\mathcal E_e=8\sum_{j}\sum_{ck=j}\mu(c)\sum_{h\neq0}(-1)^{he}\,T(h),\qquad
T(h)=\sum_{(a,b)\in\cP}\frac{G_h(q)}{2q}\,e\Big(\frac{hmb\,\inv{2a}}{q}\Big).
\end{equation}
Here $j$ runs over the odd divisors of $n$. The sums over $(a,b)$ are finite by \eqref{qmax}, and the sums over $h$ converge absolutely by Lemma~\ref{Gbounds} below.

\begin{proposition}\label{mainterm}
$\displaystyle\mathcal M_e=\frac4\pi\int_\R\Phi\Big(\frac xM\Big)\log\sqrt{n^2+x^2}\,\dd x+\frac4\pi\Big(c_0-\frac\epsilon2\log2-\lambda(n)\Big)\int_\R\Phi\Big(\frac xM\Big)\dd x+O\big(n^{1/2+\eta}\big).$
\end{proposition}
\begin{proof}
\emph{The sum over $(a,b)$.} By \eqref{Me} and the definition of $F_q$,
\[
\mathcal M_e=4\sum_j\sum_{ck=j}\mu(c)\int_\R\Phi\Big(\frac{jy}M\Big)\Sigma(y)\dd y,
\qquad \Sigma(y)=\sum_{(a,b)\in\cP}\frac{W(K_yq)}q,\quad K_y=\frac{k2^{\epsilon/2}}{\sqrt{m^2+y^2}}.
\]
By Lemma~\ref{harmonic}(ii),
\[
\Sigma(y)=\frac1\pi\log\frac{\sqrt{m^2+y^2}}{k2^{\epsilon/2}}+\gamma_0+O\big(K_y^{1/2}\log n\big),\qquad K_y\le \frac{2k}{m}\le\frac{2j^2}n.
\]

\emph{The remainder.} It contributes
\[
\ll\sum_{j\mid n}\tau(j)\frac Mj\cdot\frac{j}{\sqrt n}\log n\ll Mn^{-1/2+\eta}\le n^{1/2+\eta}.
\]

\emph{The main part.} Substitute $x=jy$ and use $\sqrt{m^2+y^2}=\sqrt{n^2+x^2}/j$. The main part becomes
\[
4\sum_{j}\frac1j\int_\R\Phi\Big(\frac xM\Big)\sum_{ck=j}\mu(c)\Big[\frac1\pi\log\frac{\sqrt{n^2+x^2}}{2^{\epsilon/2}}+\gamma_0-\frac1\pi\log(jk)\Big]\dd x .
\]

\emph{The divisor sums.} We have $\sum_{ck=j}\mu(c)=[j=1]$ and $\sum_{ck=j}\mu(c)\log(jk)=\sum_{ck=j}\mu(c)\log k=\Lambda(j)$. Moreover $\sum_{j\mid n,\ j\ \mathrm{odd}}\Lambda(j)/j=\lambda(n)$ and $\pi\gamma_0=c_0$. The proposition follows.
\end{proof}

\section{The error term}\label{sec:err}
In this section we prove:

\begin{proposition}\label{errorprop}
$\mathcal E_e\ll n^{3/4+\eta}+n^{5/4+\eta}M^{-1/2}$.
\end{proposition}

The proof proceeds in five steps:
\begin{enumerate}
\item bounds for $G_h$ (Lemma~\ref{Gbounds});
\item a dyadic decomposition of $q$;
\item reciprocity and Poisson summation in $b$, which produce Kloosterman sums (Lemmas~\ref{recipstep}--\ref{Kbound});
\item the range of the dual variable $l$ (Lemma~\ref{lrange});
\item Plancherel and assembly.
\end{enumerate}
Throughout, $j,c,k,m$ are fixed as in Section~\ref{sec:main}.

\begin{lemma}\label{Gbounds}
\begin{enumerate}[label=(\roman*)]
\item $F_q^{(\nu)}(y)\ll_\nu (j/M)^\nu$, $\int|F_q|\ll M/j$, and $G_h(q)\ll_A \frac Mj\big(1+\frac{|h|M}{jq}\big)^{-A}$.
\item Let $Q\ge1$, $q\in[Q/2,2Q]$ and $|h|\le H_Q:=8Qjn^\eta/M$. Then $\partial_q^\mu G_h(q)\ll_\mu\frac Mj\big(\frac{n^\eta}Q\big)^\mu$.
\end{enumerate}
\end{lemma}
\begin{proof}
\emph{(i)} Put $t(y)=kq2^{\epsilon/2}(m^2+y^2)^{-1/2}$. Then $t^{(\nu)}(y)\ll_\nu t(y)\,m^{-\nu}$, and $t\le2$ on the support of $W'$. By the chain rule, $\partial_y^\nu W(t(y))\ll m^{-\nu}\le(j/M)^\nu$, using $m=n/j\ge M/j$; the factor $\Phi(jy/M)$ is clear. The support of $F_q$ has length $\le4M/j$. Integrating by parts $A$ times gives the bound for $G_h$.

\emph{(ii)} In $\partial_q$, the factor $W(t)$ contributes $W'(t)t/q\ll1/q$. The factor $e(-hy/(2q))$ contributes $\pi hy/q^2$, and $|hy|/q\le H_Q(2M/j)/(Q/2)=32n^\eta$. Differentiating $\mu$ times under the integral gives the claim.
\end{proof}

\subsection*{Step 2: dyadic decomposition}
Let $\omega$ be smooth with support in $[\frac12,2]$ and $\sum_{r\ge0}\omega(q/2^r)=1$ for $q\ge1$. By Lemma~\ref{Gbounds}(i), the terms with $q\asymp Q$ and $|h|>H_Q$ contribute a negligible amount to \eqref{Ee}. By \eqref{qmax}, only $Q\le10m/k$ occur. Hence
\begin{equation}\label{EeQ}
\mathcal E_e=8\sum_j\sum_{ck=j}\mu(c)\sum_{\substack{Q=2^r\\ Q\le10m/k}}\ \sum_{0<|h|\le H_Q}(-1)^{he}\,T_Q(h)+O(n^{-1}),
\end{equation}
where $T_Q(h)$ is $T(h)$ with the extra factor $\omega(q/Q)$.

\subsection*{Step 3: symmetrisation, reciprocity, Poisson in $b$}
The summand of $T_Q(h)$ depends only on the ideal $(a+bi)$. Indeed, by Lemma~\ref{cong}, the residue $mb\inv a\bmod q$, and hence the phase in Lemma~\ref{poissonlemma}, is determined by the ideal. So we may sum over all generators of primitive odd ideals and divide by $4$. The pairs $(a,b)$ and $(-a,-b)$ give the same ideal, and $a=0$ occurs only for $(0,\pm1)$, where $q=1$. Hence
\begin{equation}\label{TQ}
T_Q(h)=\frac12\sum_{a\ge1}\ \sum_{b\in\mathcal B_a}\frac{\omega(q/Q)G_h(q)}{2q}\,e\Big(\frac{hmb\,\inv{2a}}{q}\Big)+\frac{\omega(1/Q)}{4}G_h(1),
\end{equation}
where $\mathcal B_a=\{b\in\Z:(a,b)=1,\ b\not\equiv a\bmod2\}$. By Lemma~\ref{Gbounds}(i),
\[
\sum_{h\ne0}|G_h(1)|\ll\frac Mj\sum_{h\neq0}\Big(1+\frac{|h|M}j\Big)^{-2}\ll1 .
\]
So the last term in \eqref{TQ} contributes $\ll\tau(n)^2\ll n^\eta$ to \eqref{EeQ}. Moreover, $q\ge a^2$ and $q\le2Q$ force $a\le\sqrt{2Q}$.

\begin{lemma}\label{recipstep}
Let $a\ge1$, $b\in\mathcal B_a$ and $q=a^2+b^2$. Then $(2a,q)=1$ and
\[
e\Big(\frac{hmb\,\inv{2a}}{q}\Big)=e\Big(-\frac{hmb\,\inv q}{2a}\Big)\,e\Big(\frac{hmb}{2aq}\Big),
\]
where $\inv q$ is the inverse of $q$ modulo $2a$. The first factor on the right depends only on $b\bmod 2a$.
\end{lemma}
\begin{proof}
The claim $(2a,q)=1$ holds because $q$ is odd and $(a,q)=(a,b^2)=1$. The identity is \eqref{recip} with $u=2a$, $v=q$, multiplied by $hmb$. Finally $q\bmod 2a$ depends only on $b\bmod 2a$, because $(b+2a)^2\equiv b^2\bmod 2a$.
\end{proof}

For $a\le\sqrt{2Q}$ and $0<|h|\le H_Q$ define the smooth, compactly supported function
\[
g(b)=g_{a,Q,h}(b)=\frac{\omega(q/Q)G_h(q)}{2q}\,e\Big(\frac{hmb}{2aq}\Big),\qquad q=a^2+b^2\ (b\in\R),
\]
and the complete exponential sum
\[
K_a(h,l)=\sum_{\substack{\beta\bmod 2a\\ \beta\in\mathcal B_a}}e\Big(\frac{-hm\beta\,\inv{q_\beta}+l\beta}{2a}\Big),\qquad q_\beta=a^2+\beta^2.
\]

\begin{lemma}\label{poissonb}
$\displaystyle\sum_{b\in\mathcal B_a}g(b)\,e\Big(-\frac{hmb\,\inv q}{2a}\Big)=\frac1{2a}\sum_{l\in\Z}K_a(h,l)\,\hat g\Big(\frac{l}{2a}\Big).$
\end{lemma}
\begin{proof}
Split $b$ into residue classes $\beta\bmod 2a$ (the set $\mathcal B_a$ is $2a$-periodic) and apply \eqref{poisson} in each class.
\end{proof}

\begin{lemma}\label{Kbound}
$|K_a(h,l)|\le4\tau(a)\,(hm,a)^{1/2}a^{1/2}$.
\end{lemma}
\begin{proof}
\emph{Case $a$ even.} Then $\mathcal B_a$ consists of the residues coprime to $2a$. Since $2a\mid a^2$, we have $q_\beta\equiv\beta^2\bmod 2a$, so $\beta\inv{q_\beta}\equiv\inv\beta\bmod2a$ and $K_a(h,l)=S(l,-hm;2a)$. By \eqref{weil} this is at most $\tau(2a)(l,hm,2a)^{1/2}(2a)^{1/2}\le 4\tau(a)(hm,a)^{1/2}a^{1/2}$.

\emph{Case $a$ odd.} Then $\beta=2\beta'$ with $\beta'$ running over the reduced residues modulo $a$. Let $z\equiv\beta\inv{q_\beta}\bmod 2a$. Then $z\equiv0\bmod2$ and $z\equiv\inv\beta\bmod a$, so $z/(2a)\equiv\inv{4\beta'}/a\bmod1$, with the inverse taken modulo $a$. Also $l\beta/(2a)=l\beta'/a$. Hence $K_a(h,l)=S(l,-hm\inv4;a)$, and \eqref{weil} gives $\le\tau(a)(l,hm,a)^{1/2}a^{1/2}$.
\end{proof}

\subsection*{Step 4: the range of $l$}
We use the following first-derivative estimate.

\begin{lemma}[{\cite[Lemma~8.1]{BKY}}]\label{BKY}
Let $w$ be smooth with support in $[\alpha,\beta]$ and $w^{(\nu)}\ll_\nu X\,U^{-\nu}$. Let $\theta$ be smooth with $|\theta'|\ge R>0$ on $[\alpha,\beta]$ and $\theta^{(\nu)}\ll_\nu Y\,Q_0^{-\nu}$ for $\nu\ge2$, where $Y\ge1$. Then
\[
\int w(t)e^{i\theta(t)}\dd t\ll_A(\beta-\alpha)X\big[(Q_0R/\sqrt Y)^{-A}+(RU)^{-A}\big].
\]
\end{lemma}

\begin{lemma}\label{lrange}
Let $a\le\sqrt{2Q}$ and $0<|h|\le H_Q$. Then $\hat g(l/(2a))$ is negligible for $|l|\ge 8|h|m/Q+n^{3\eta}$.
\end{lemma}
\begin{proof}
We have $\hat g(l/2a)=\int w(b)e^{i\theta(b)}\dd b$, where
\[
w(b)=\frac{\omega(q/Q)G_h(q)}{2q},\qquad \theta(b)=2\pi\Big(\frac{hmb}{2aq}-\frac{lb}{2a}\Big).
\]
On the support, $q\in[Q/2,2Q]$ and $|b|\le\sqrt{2Q}$. The support of $w$ consists of at most two intervals, at whose endpoints $w$ vanishes to infinite order. We apply Lemma~\ref{BKY} to each of them, after verifying its hypotheses.

\emph{Amplitude.} By Lemma~\ref{Gbounds}(ii), the bounds for $\omega$, and Faà di Bruno's formula applied to $q=a^2+b^2$ (where $\partial_b q=2b\ll Q^{1/2}$ and $\partial_b^2q=2$), we get
\[
w^{(\nu)}\ll X(n^\eta Q^{-1/2})^\nu,\qquad X\asymp\frac{M}{jQ},\quad U=Q^{1/2}n^{-\eta}.
\]

\emph{First derivative.} Since $\partial_b\frac{b}{a^2+b^2}=\frac{a^2-b^2}{q^2}$ and $|a^2-b^2|\le q$,
\[
\Big|\frac{\theta'(b)}{2\pi}\Big|=\frac1{2a}\Big|\frac{hm(a^2-b^2)}{q^2}-l\Big|\ge\frac1{2a}\Big(|l|-\frac{2|h|m}{Q}\Big)\ge\frac{|l|}{3a},
\]
so we may take $R=2\pi|l|/(3a)$.

\emph{Higher derivatives.} Since $b/(a^2+b^2)=-\Im\,(a+ib)^{-1}$, we have $\partial_b^\nu\frac{b}{a^2+b^2}\ll_\nu q^{-(1+\nu)/2}$. Hence for $\nu\ge2$,
\[
\theta^{(\nu)}\ll\frac{|h|m}{a}Q^{-(1+\nu)/2},\qquad Q_0=Q^{1/2},\quad Y=\max\Big(1,\frac{|h|m}{aQ^{1/2}}\Big).
\]

\emph{The two terms in Lemma~\ref{BKY}.}
\begin{itemize}
\item $RU\ge|l|n^{-\eta}\ge n^{2\eta}$, using $a\le\sqrt{2Q}$.
\item If $Y=1$, then $Q_0R\ge|l|\ge n^{3\eta}$.
\item Otherwise $Q_0R/\sqrt Y=\frac{2\pi}3|l|Q^{3/4}(a|h|m)^{-1/2}\ge|l|(Q/|h|m)^{1/2}\ge|l|^{1/2}\sqrt8\ge n^{3\eta/2}$, using $|l|\ge8|h|m/Q$.
\end{itemize}
Lemma~\ref{BKY} gives $\hat g(l/2a)\ll_A Q^{1/2}Xn^{-A\eta}$, which is negligible.
\end{proof}

\subsection*{Step 5: Plancherel and conclusion}

\begin{lemma}\label{planch}
For $g\in L^2(\R)$ supported in an interval of length $\ell$, and $c>0$,
\[
\sum_{l\in\Z}\Big|\hat g\Big(\frac lc\Big)\Big|^2\le(\ell+c)\int_\R|g|^2 .
\]
\end{lemma}
\begin{proof}
The $c$-periodic function $P(x)=\sum_kg(x+kc)$ has Fourier coefficients $c^{-1}\hat g(l/c)$, so by Parseval $\sum_l|\hat g(l/c)|^2=c\int_0^c|P|^2$. At most $\ell/c+1$ translates overlap at any point, so Cauchy--Schwarz gives $|P(x)|^2\le(\ell/c+1)\sum_k|g(x+kc)|^2$.
\end{proof}

\begin{proof}[Proof of Proposition~\ref{errorprop}]
Fix $Q$, $0<|h|\le H_Q$ and $1\le a\le\sqrt{2Q}$.

\emph{The $l$-sum.} By Lemma~\ref{lrange} only $|l|\le L:=8|h|m/Q+n^{3\eta}$ matter. The function $g$ is supported in $|b|\le\sqrt{2Q}$ and satisfies $|g|\ll M/(jQ)$. By Cauchy--Schwarz and Lemma~\ref{planch} with $c=2a\le2\sqrt{2Q}$,
\[
\sum_l\Big|\hat g\Big(\frac l{2a}\Big)\Big|\le(2L+1)^{1/2}\Big(6\sqrt Q\int|g|^2\Big)^{1/2}+\text{negl.}\ll n^{2\eta}\Big(1+\frac{|h|m}Q\Big)^{1/2}\frac Mj\,Q^{-1/2}.
\]

\emph{The $a$-sum.} Combine this with Lemmas~\ref{recipstep}, \ref{poissonb} and \ref{Kbound}. Using
\[
\sum_{a\le A}\tau(a)(v,a)^{1/2}a^{-1/2}\le\Big(\max_{a\le A}\tau(a)\Big)\sum_{d\mid v}d^{1/2}\sum_{a'\le A/d}(da')^{-1/2}\ll n^{\eta}\tau(v)A^{1/2}
\]
with $v=|h|m\le n^2$, the first part of \eqref{TQ} is
\[
\ll\sum_{a\le\sqrt{2Q}}\frac1a\,\tau(a)(hm,a)^{1/2}a^{1/2}\,n^{2\eta}\Big(1+\frac{|h|m}Q\Big)^{1/2}\frac Mj\,Q^{-1/2}
\ll n^{4\eta}\,\frac Mj\,Q^{-1/4}\Big(1+\frac{|h|m}Q\Big)^{1/2}.
\]

\emph{The $h$-sum.} Using $\sum_{|h|\le H}(1+|h|m/Q)^{1/2}\ll H+H^{3/2}(m/Q)^{1/2}$ with $H=H_Q=8Qjn^\eta/M$, we obtain
\[
\sum_{0<|h|\le H_Q}|T_Q(h)|\ll n^{6\eta}\Big(Q^{3/4}+Q^{3/4}\,\frac{j^{1/2}m^{1/2}}{M^{1/2}}\Big)+n^{\eta}.
\]

\emph{The $Q$- and $j$-sums.} Sum over the dyadic $Q\le10m/k\le10n/j$. Then sum over $j\mid n$ and $ck=j$, using $m=n/j$:
\[
\mathcal E_e\ll n^{7\eta}\sum_{j\mid n}\tau(j)\Big(\Big(\frac nj\Big)^{3/4}+\frac{n^{5/4}}{j^{3/4}M^{1/2}}\Big)\ll n^{9\eta}\Big(n^{3/4}+\frac{n^{5/4}}{M^{1/2}}\Big).
\]
Replacing $\eta$ by $\eta/9$ completes the proof.
\end{proof}

\section{Proof of Theorem~\ref{mainthm} and numerical check}\label{sec:num}
By \eqref{Se} and Lemma~\ref{poissonlemma}, $S_e=\mathcal M_e+\mathcal E_e$. Theorem~\ref{mainthm} follows from Propositions~\ref{mainterm} and \ref{errorprop}; the term $O(n^{1/2+\eta})$ is absorbed into $O(n^{3/4+\eta})$.

Table~\ref{tab} compares both sides of \eqref{mainformula} numerically, without the $O$-term. Here $\Phi(t)=\exp(-1/((t-1)(2-t)))$ on $(1,2)$ and $\Phi=0$ otherwise. The left-hand side was computed exactly by sieving $n^2+x^2$, and the integrals by Riemann sums, which are spectrally accurate for smooth $\Phi$. In all cases the difference is of size $O(\sqrt M)$, far below the bound proved.

\begin{table}[ht]
\centering
\begin{tabular}{rlrcrrr}
\toprule
$n$ & factorisation & $M$ & $e$ & LHS of \eqref{mainformula} & main terms & difference\\
\midrule
1000003 & prime & 100000 & 0 & 14373.2 & 14339.5 & 33.7\\
1000003 & prime & 100000 & 1 & 14018.5 & 14029.3 & $-10.8$\\
1000003 & prime & 1000003 & 0 & 148553.5 & 148559.6 & $-6.1$\\
4849845 & $3\cdot5\cdot7\cdot11\cdot13\cdot17\cdot19$ & 484984 & 0 & 68993.6 & 69005.7 & $-12.1$\\
4849845 & $3\cdot5\cdots19$ & 484984 & 1 & 67507.6 & 67501.2 & 6.4\\
52026975 & $3^35^27^211^213$ & 2000000 & 0 & 327574.3 & 327500.6 & 73.7\\
52026975 & $3^35^27^211^213$ & 2000000 & 1 & 321311.4 & 321296.5 & 14.9\\
49999991 & prime & 2000000 & 1 & 350378.3 & 350447.7 & $-69.5$\\
2000006 & $2\cdot1000003$ & 300000 & 1 & 44946.3 & 44916.3 & 30.0\\
\bottomrule
\end{tabular}
\caption{The two sides of \eqref{mainformula}. For $n=4849845$ and $n=52026975$ the arithmetic secondary term $-\frac4\pi\lambda(n)\int\Phi$ amounts to about $11\%$ and $9\%$ of the left-hand side, while the observed discrepancy is below $0.03\%$. Corollary~\ref{cor2} was checked in the same way, including $n=2^{23}$.}\label{tab}
\end{table}

\section{Remarks on further improvements}\label{sec:rem}
\begin{enumerate}[leftmargin=*]
\item \emph{Where the range $5/6$ comes from.} The error term is a sum of Kloosterman sums $S(l,-hm;a)$, possibly with $2a$ in place of $a$, over moduli $a\le\sqrt{2Q}\ll\sqrt{n}$ and frequencies $0<|h|\ll Qj/M$, $|l|\ll|h|m/Q$, weighted by the oscillatory integrals $\hat g$. We used Weil's bound for each modulus and Plancherel in $l$. Bounding $\hat g$ pointwise instead, by $\int|g|$, would only give the range $M\ge n^{7/8+\varepsilon}$ of \cite{B24}.
\item \emph{Averaging over the modulus (open).} Any further improvement must exploit cancellation in the sum over the modulus $a$. Here the natural parameter of the Kuznetsov formula is $\sqrt{|l|hm}/a\asymp hm/Q\le n^{1+\eta}/M$. The weights $\hat g(l/2a)$ oscillate in $a$ on this same scale, and the first argument $hm$ is large, of size $\asymp n|h|/j$. The available general bounds for sums of Kloosterman sums \cite{DeIw} require non-oscillating weights, so they lose powers of $n/M$; moreover the multiplier $m\asymp n$ costs a factor $n^{\theta}$, with $\theta=7/64$ \cite{Kim}, in the large sieve for Fourier coefficients. We do not know how to obtain a gain this way.

\emph{Heuristic, not proved.} If the sums over $a$ had square-root cancellation beyond Weil's bound, the error term would become $\ll n^{1+\eta}M^{-1/2}$, and the asymptotic formula would hold for $M\ge n^{2/3+\varepsilon}$.
\item \emph{The true error term.} The numerical data in Table~\ref{tab} are consistent with an error of size $M^{1/2+o(1)}$.
\end{enumerate}

\end{document}